\documentclass[11pt]{amsart}

\usepackage{amsfonts,amsmath,amssymb}
\usepackage{pinlabel}
\usepackage{amsthm}
\usepackage{latexsym}
\usepackage{graphicx}
\usepackage{psfrag}
\usepackage{color}

\copyrightinfo{2006}{American Mathematical Society}

\newtheorem{theorem}{Theorem}[section]
\newtheorem{lemma}[theorem]{Lemma}

\theoremstyle{definition}

\newtheorem{remark}[theorem]{Remark}

\newcommand{\Nbd}{\operatorname{Nbd}}
\newcommand{\cl}{\operatorname{cl}}

\numberwithin{equation}{section}

\begin{document}

\title[Genus two Goeritz group of $\mathbb S^2 \times \mathbb S^1$]{The genus two Goeritz group of $\mathbb S^2 \times \mathbb S^1$}

\author{Sangbum Cho}\thanks{The first-named author is supported by Basic Science Research Program through the National Research Foundation of Korea(NRF) funded by the Ministry of Education, Science and Technology (2012006520).}
\address{Department of Mathematics Education, Hanyang University, Seoul 133-791,
Korea}
\email{scho@hanyang.ac.kr}

\author{Yuya Koda}\thanks{The second-named author is supported in part by
Grant-in-Aid for Young Scientists (B) (No. 20525167), Japan Society for the Promotion of Science, and by
JSPS Postdoctoral Fellowships for Research Abroad.}

\address{
Mathematical Institute \newline
\indent Tohoku University, Sendai, 980-8578, Japan \newline
\indent and \newline
\indent (Temporary) Dipartimento di Matematica  \newline
\indent Universit\`{a} di Pisa, Largo Bruno Pontecorvo 5, 56127 Pisa, Italy}
\email{koda@math.tohoku.ac.jp}


\date{\today}

\begin{abstract}
The genus-$g$ Goeritz group is the group of isotopy classes of orientation-preserving homeomorphisms of a closed orientable $3$-manifold that preserve a given genus-$g$ Heegaard splitting of the manifold.
In this work, we show that the genus-$2$ Goeritz group of $\mathbb S^2 \times \mathbb S^1$ is finitely presented, and give its explicit presentation.
\end{abstract}

\maketitle

\section{Introduction}
\label{sec:intro}
Given a genus-$g$ Heegaard splitting of a closed orientable $3$-manifold, the genus-$g$ {\it Goeritz group} is the group of isotopy classes of orientation-preserving homeomorphisms of the manifold preserving each of the handlebodies of the splitting setwise.
When a manifold admits a unique Heegaard splitting of genus $g$ up to isotopy, we might define the genus-$g$ Goeritz group of the manifold without mentioning a specific splitting.
For example, the $3$-sphere, $\mathbb S^2 \times \mathbb S^1$ and lens spaces are known to be such manifolds from \cite{W}, \cite{B} and \cite{B-O}.

It is natural to study the structures of Goeritz groups and to ask if they are finitely generated or presented, and so finding their generating sets or presentations has been an interesting problem.
But the generating sets or the presentations of those groups have been obtained only for few manifolds with their splittings of small genus.
A finite presentation of the genus-$2$ Goeritz group of the $3$-sphere was obtained from the works of \cite{Ge}, \cite{Sc}, \cite{Ak} and \cite{C}, and recently of each lens space $L(p, 1)$ in \cite{C2}.
In addition, finite presentations of the genus-$2$ Goeritz groups of other lens spaces are given in \cite{C-K2}.
Also a finite generating set of the genus-$3$ Goeritz group of the $3$-torus was obtained in \cite{J}.
For the higher genus Georitz groups of the $3$-sphere and lens spaces, it is conjectured that they are all finitely presented but it is still known to be an open problem.

In this work, we show that the genus-$2$ Goeritz group of $\mathbb S^2 \times \mathbb S^1$ is finitely presented, by giving its explicit presentation as follows.

\begin{theorem}
The genus-$2$ Goeritz group $\mathcal G$ has the presentation
$$\langle ~\epsilon~ \rangle \oplus \langle ~\alpha ~|~ \alpha^2 =1~\rangle \oplus \langle ~ \beta, \gamma, \sigma ~|~ \gamma^2 = \sigma^2 = (\gamma \beta \sigma)^2 = 1 ~ \rangle.$$
\label{thm:presentation}
\end{theorem}

The generators $\alpha$, $\beta$, $\gamma$ and $\sigma$ are illustrated in Figure \ref{fig_generators} as orientation-preserving homeomorphisms of a Heegaard surface which can extend to homeomorphisms of the whole $\mathbb S^2 \times \mathbb S^1$.
The generator $\epsilon$ is the Dehn twist about the circle $\partial E_0$ in the figure.
In Section \ref{sec:primitive_disks}, we describe those generators in detail.

To find the presentation of $\mathcal G$, we generalize the method developed in \cite{C}.
We construct a tree on which the group $\mathcal G$ acts such that the quotient of the tree by the action is a single edge, and then find the presentations of the stabilizer subgroups of the edge and  each of end vertices.

Throughout the paper, we denote by $(V, W; \Sigma)$ the genus-$2$ Heegaard splitting of $\mathbb S^2 \times \mathbb S^1$.
That is, $V$ and $W$ are genus-$2$ handlebodies such that $V \cup W = \mathbb S^2 \times \mathbb S^1$ and $\partial V = \partial W = \Sigma$.
All disks in a handlebody are always assumed to be properly embedded and their intersections are transverse and minimal up to isotopy.
In particular, if two disks intersect each other, then the intersection is a collection of pairwise disjoint arcs that are properly embedded in each disk.
Finally, $\Nbd(X)$ will denote a regular neighborhood of $X$ and $\cl(X)$ the closure of $X$ for a subspace $X$ of a polyhedral space, where the ambient space will always be clear from the context.

\section{Primitive disks in a handlebody}
\label{sec:primitive_disks}

Recall that $V$ is a genus-$2$ handlebody in $\mathbb S^2 \times \mathbb S^1$ whose exterior is the handlebody $W$.
A non-separating disk $E_0$ in $V$ is called a {\it reducing disk} if there exists a disk $E'_0$ in $W$ such that $\partial E_0 = \partial E'_0$.
The disk $E'_0$ is also a reducing disk in $W$.
An essential disk $E$ in $V$ is called {\it primitive} if there exists an essential disk $E'$ in $W$ such that $\partial E$ meets $\partial E'$ in a single point.
Such a disk $E'$ is called a {\it dual disk} of $E$, and $E'$ is also primitive in $W$ with its dual disk $E$.
Primitive disks are necessarily non-separating.
We remark that $\cl(V - \Nbd(E))$ and $W \cup \Nbd(E)$ (and $V \cup \Nbd(E')$ and $\cl(W - \Nbd(E'))$, respectively) are solid tori, which form a genus-$1$ Heegaard splitting of $\mathbb S^2 \times \mathbb S^1$.
It follows that, for a meridian disk $E_0$ of the solid torus $\cl(V - \Nbd(E))$, there exists
a meridian disk $E'_0$ of the solid torus $W \cup \Nbd(E)$ satisfying $\partial E_0 = \partial E'_0$.
In particular, we can find such disks $E_0$ and $E'_0$ so that they are disjoint from the $3$-ball $\Nbd(E \cup E')$.
Therefore, once we have a primitive disk $E$ with its dual disk $E'$, there exist reducing disks $E_0$ in $V$ and $E'_0$ in $W$ disjoint from $E \cup E'$ such that $\partial E_0 = \partial E'_0$.
See Figure \ref{fig_splitting}.
\begin{figure}
\centering
\labellist
\pinlabel {\small $E_0$} [B] at 20 33
\pinlabel {\small $\partial E'$} [B] at 113 66
\pinlabel {\small $E$} [B] at 143 30

\pinlabel {\small $E'_0$} [B] at 233 33
\pinlabel {\small $\partial E$} [B] at 325 66
\pinlabel {\small $E'$} [B] at 356 30

\pinlabel {\large $V$} [B] at 171 5
\pinlabel {\large $W$} [B] at 390 5
\endlabellist
\includegraphics[width=0.8\textwidth]{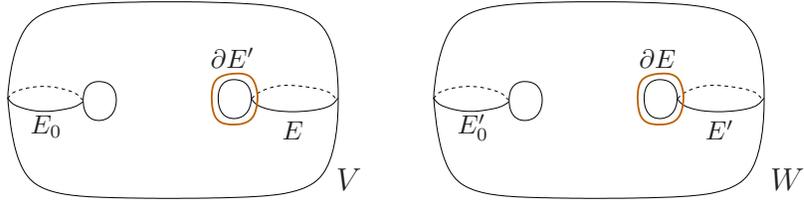}
\caption{
A genus two Heegaard splitting of $\mathbb S^2 \times \mathbb S^1$.}
\label{fig_splitting}
\end{figure}
We call a pair of disjoint, non-isotopic primitive disks in $V$ a {\it primitive pair} of $V$, and if a disk in $W$ is a dual disk of each the two disks of the pair, then call it simply a {\it common dual disk} of the pair.

\smallskip

We first introduce six elements $\alpha$, $\beta$, $\beta'$, $\gamma$, $\sigma$ and $\epsilon$ of the group $\mathcal G$,
which will turn out to form a generating set of $\mathcal G$.
These elements can be described as orientation-preserving homeomorphisms of the surface
$\Sigma$ which extend to homeomorphisms of the whole $\mathbb S^2 \times \mathbb S^1$ preserving each of $V$ and $W$ setwise.
Fix a primitive pair $\{D, E\}$ of $V$ and a primitive pair $\{D', E'\}$ of common dual disks of $D$ and $E$.
(The existence of such disks will be shown in Lemmas \ref{lem:common_dual} and \ref{lem:two_common_duals}.)
Figures \ref{fig_generators} (a) and (b) illustrate these disks with the unique reducing disks $E_0$ in $V$ and $E'_0$ in $W$ disjoint from them.
Notice that $\partial E_0 = \partial E'_0$. 
(The existence and uniqueness of such reducing disks will be shown in Lemma \ref{lem:reducing_disk_and_primitive_disk}.)

\begin{figure}
\centering
\labellist
\pinlabel {\small $E_0$} [B] at 31 252
\pinlabel {\small $D$} [B] at 83 252
\pinlabel {\small $E$} [B] at 140 252
\pinlabel {\small $E'_0$} [B] at 256 252
\pinlabel {\small $D'$} [B] at 310 250
\pinlabel {\small $E'$} [B] at 365 252
\pinlabel {\small $\partial E'$} [B] at 113 245
\pinlabel {\small $\partial E$} [B] at 340 245
\pinlabel {\small $\partial D'$} [B] at 67 232
\pinlabel {\small $\partial D$} [B] at 294 232
\pinlabel {\small $\alpha$} [B] at 180 265
\pinlabel {\small $\beta'$} [B] at 222 263
\pinlabel {\small $\beta$} [B] at 406 263
\pinlabel {\small $E_0$} [B] at 32 85
\pinlabel {\small $D$} [B] at 107 87
\pinlabel {\small $E$} [B] at 146 87
\pinlabel {\small $E'_0$} [B] at 253 84
\pinlabel {\small $D'$} [B] at 326 84
\pinlabel {\small $E'$} [B] at 366 84
\pinlabel {\small $\alpha$} [B] at 184 98
\pinlabel {\small $\alpha$} [B] at 402 96
\pinlabel {\small $\sigma$} [B] at 125 20
\pinlabel {\small $\gamma$} [B] at 344 17
\pinlabel {\Large $V$} [B] at 160 220
\pinlabel {\Large $W$} [B] at 390 220
\pinlabel {\Large $V$} [B] at 152 41
\pinlabel {\Large $W$} [B] at 372 40
\pinlabel {\large (a)} [B] at 87 200
\pinlabel {\large (b)} [B] at 313 200
\pinlabel {\large (c)} [B] at 87 12
\pinlabel {\large (d)} [B] at 313 12
\endlabellist
\includegraphics[width=1\textwidth]{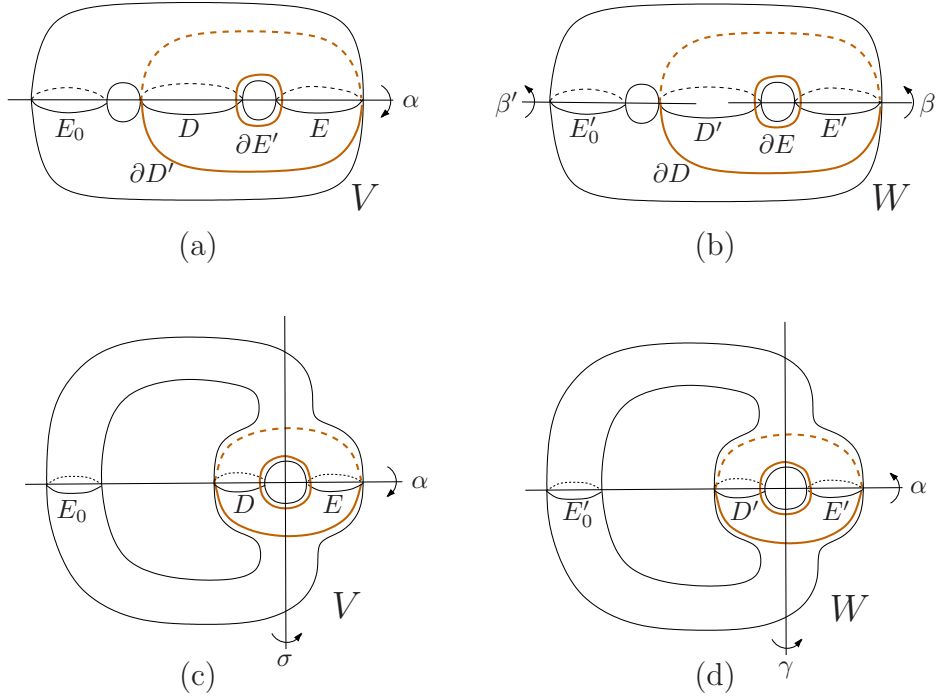}
\caption{
The elements of the Goeritz group $\mathcal G$.}
\label{fig_generators}
\end{figure}

In Figure \ref{fig_generators} (a), the order two element $\alpha$ of $\mathcal G$ is described as a hyperelliptic involution of $\Sigma$.
The elements $\beta$ and $\beta'$ described in Figure \ref{fig_generators} (b) are
half-Dehn twists about a separating loop in $\Sigma$
disjoint from $\partial E \cup \partial E'$ and from $\partial E_0 = \partial E'_0$.
This loop meets each of
$\partial D$ and $\partial D'$ transversely at two points,
and bounds a disk in each of $V$ and $W$. The elements $\beta$ and $\beta'$ have infinite order and satisfy $\beta \beta' = \alpha$.
The order two element $\gamma$ in Figure \ref{fig_generators} (d) exchanges $D'$ and $E'$ while fixing each of $E$ and $E_0$.
Notice that $D$ and $\gamma(D) (= \beta (D))$ are the two unique common dual disks of the primitive pair $\{D', E'\}$ disjoint from $E$. (The uniqueness will be shown in Lemma \ref{lem:two_common_duals}.)
In the same manner, the order two element $\sigma$ in Figure \ref{fig_generators} (c) exchanges $D$ and $E$ and fixes each of $E'$ and $E'_0$.
Also, $D'$ and $\sigma(D') (= \beta^{-1}(D'))$ are the two unique common dual disks of the primitive pair $\{D, E\}$ disjoint from $E'$.
Finally, the element $\epsilon$ is a Dehn twist about $\partial E_0 = \partial E'_0$.
Notice that all of those elements preserve each of $E_0$ and $E'_0$.
In addition, we define $\tau = \gamma \beta$ and $\tau' = \tau \sigma$ for later in the argument.
The element $\tau$ has infinite order and preserves each of $D$ and $E$, but sends $E'$ to $D'$.
The element $\tau'$ has order two, and exchanges each of $D$ and $E$, and $D'$ and $E'$ respectively.

\smallskip

In the remaining of the section, we develop several properties of primitive disks and reducing disks, which will be used in the next sections.
The boundary circle of any essential disk in $V$ represents an element of $\pi_1(W)$, a free group of rank two.
In particular, if $E_0$ is a reducing disk in $V$, then $\partial E_0$ represents the trivial element of $\pi_1(W)$.
For primitive disks, we have the following intepretation, which is a direct consequence of \cite{Go}.

\begin{lemma}
Let $E$ be a non-separating disk in $V$.
Then $E$ is primitive if and only if $\partial E$ represents a primitive element of $\pi_1(W)$.
\label{lem:primitive}
\end{lemma}

Let $F$ and $G$ be essential disks in $W$ such that $F \cup G$ cuts $W$ into a $3$-ball.
Assign symbols $x$ and $y$ to $\partial F$ and $\partial G$ respectively after fixing orientations of $\partial F$ and $\partial G$.
Let $l$ be any oriented simple closed curve in $\partial W$ such that $l$ intersects $\partial F \cup \partial G$ transversally and minimally.
Then $l$ determines a word in terms of $x$ and $y$ which can be read off from the intersections of $l$ with $\partial F$ and $\partial G$.
Thus $l$ represents an element of the free group $\pi_1(W) = \langle x, y\rangle$.
In this set up, the following is a simple criterion for triviality and primitiveness of the elements represented by $l$, which is found in Lemma 2.2 in \cite{C} with its proof.

\begin{lemma}
If a word determined by $l$ contains a sub-word of the form $yxy^{-1}$ after a suitable choice of orientations, then $l$ represents a non-trivial, non-primitive element of $\pi_1(W)$.
\label{lem:non_primitive}
\end{lemma}

The idea of the proof is that, once a word determined by $l$ contains $yxy^{-1}$, any word determined by $l$ is cyclically reduced and so nonempty, and any cyclically reduced word containing both $y$ and $y^{-1}$ cannot represent a primitive element of $\langle x, y \rangle$.

Let $D$ and $E$ be essential disks in $V$, and suppose that $D$ intersects $E$ transversely and minimally.
Let $C \subset D$ be a disk cut off from $D$ by an outermost arc $\beta$ of $D \cap E$ in $D$ such that $C \cap E= \beta$.
We call such a $C$ an {\it outermost sub-disk} of $D$ cut off by $D \cap E$.
The arc $\beta$ cuts $E$ into two disks, say $H$ and $K$.
Then we have two essential disks $E_1$ and $E_2$ in $V$ which are isotopic to disks $H \cup C$ and $K \cup C$ respectively.
We call $E_1$ and $E_2$ the {\it disks from surgery} on $E$ along the outermost sub-disk $C$ of $D$ cut off by $D \cap E$.
Observe that each of $E_1$ and $E_2$ has fewer arcs of intersection with $D$ than $E$ had, since at least the arc $\beta$ no longer counts.

Since $E$ and $D$ are assumed to intersect minimally, $E_1$ (and $E_2$) is isotopic to neither $E$ nor $D$.
In particular, if $E$ is non-separating, then the resulting disks $E_1$ and $E_2$ are both non-separating and they are not isotopic to each other
because, after isotopying $E_1$ and $E_2$ away from $E$, both of them are meridian disks of the solid torus $V$ cut off by $E$, and the boundary circles $\partial E_1$ and $\partial E_2$ are not isotopic to each other in the two holed torus $\partial V$ cut off by $\partial E$.

\begin{lemma}
Let $E_0$ be a reducing disk in $V$ disjoint from $E \cup E'$, where $E$ is a primitive disk in $V$ and $E'$ is a dual disk of $E$.
Let $D$ be any non-separating disk in $V$ which is not isotopic to $E_0$.
\begin{enumerate}
\item If $D$ is disjoint from $E_0$, then $D$ is a primitive disk, and hence is not a reducing disk.
\item If $D$ intersects $E_0$, then $D$ is neither a reducing disk nor a primitive disk.
\end{enumerate}
\label{lem:non_primitive_2}
\end{lemma}

\begin{proof}
Let $E'_0$ be a reducing disk in $W$ such that $\partial E_0 = \partial E'_0$.
Then $E'_0$ is also disjoint from $E \cup E'$.
Let $\Sigma'$ be the $4$-holed sphere obtained by cutting $\Sigma$ along $\partial E_0$ and $\partial E$.
We note that $\partial E'$ is an arc in $\Sigma'$ connecting two holes coming from $E$.

\smallskip
\begin{figure}
\centering
\labellist
\pinlabel {\tiny $\partial E$} [B] at 114 198
\pinlabel {\tiny $\partial E$} [B] at 114 121
\pinlabel {\tiny $\partial E$} [B] at 371 198
\pinlabel {\tiny $\partial E$} [B] at 372 120
\pinlabel {\scriptsize $\partial E'$} [B] at 125 138
\pinlabel {\scriptsize $\partial E''$} [B] at 385 138
\pinlabel {\small $\partial E_0$} [B] at 115 82
\pinlabel {\small $\partial E_0$} [B] at 373 82
\pinlabel {\small $\partial E_0$} [B] at 193 40
\pinlabel {\small $\partial E_0$} [B] at 453 40
\pinlabel {\footnotesize $\alpha$} [B] at 179 90
\pinlabel {\footnotesize $\alpha$} [B] at 436 90
\pinlabel {\small $\partial D$} [B] at 164 168
\pinlabel {\small $\partial D$} [B] at 423 167
\pinlabel {\large (a)} [B] at 118 1
\pinlabel {\large (b)} [B] at 378 1
\endlabellist
\includegraphics[width=1 \textwidth]{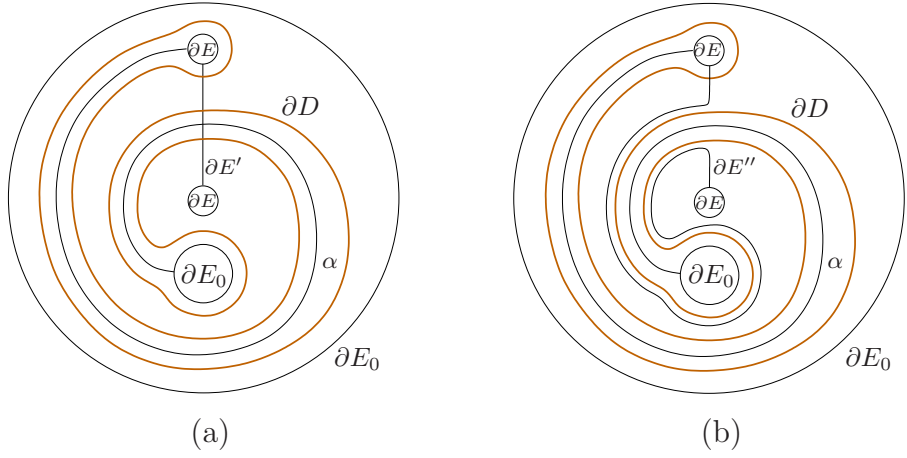}
\caption{
The $4$-holed sphere $\Sigma'$.}
\label{fig_4_holed_sphere_1}
\end{figure}
\noindent (1) Suppose that $D$ is disjoint from $E_0$.
Consider first the case that $D$ is also disjoint from $E$.
Then $D$ is determined by an arc $\alpha$ properly embedded in $\Sigma'$ connecting the holes coming from $\partial E_0$ and $\partial E$.
That is, $\partial D$ is the frontier of a regular neighborhood of the union of $\alpha$ and the two holes connected by $\alpha$ in $\Sigma'$.
See Figure \ref{fig_4_holed_sphere_1} (a).
If $\alpha$ intersects $\partial E'$, then take the sub-arc of $\alpha$ which connects $\partial E'_0 ( = \partial E_0)$ and $\partial E'$ and whose interior is disjoint from $\partial E'$.
Then the band sum of $E'_0$ and $E'$ along this sub-arc is a non-separating disk, denoted by $E''$, in $W$.
We observe that $|\alpha \cap \partial E''| < |\alpha \cap \partial E'|$, and $E''$ is again a dual disk of $E$ and is disjoint from $E_0$.
See Fig \ref{fig_4_holed_sphere_1} (b).
We can repeat this process until we find a dual disk $\widehat{E}'$ of $E$ disjoint from $E_0$ so that $\alpha$ is disjoint from $\widehat{E}'$.
Then $\partial D$ intersects $\partial \widehat{E}'$ in a single point, which implies that $D$ is
primitive.

Next, suppose that $D$ intersects $E$.
Let $C$ be any outermost sub-disk of $D$ cut off by $D \cap E$.
Then $C \cap \Sigma'$ is an arc properly embedded in $\Sigma'$ whose end points lie in a single hole coming from $\partial E$.
The arc $C \cap \Sigma'$ is determined by an arc, say $\beta$, in $\Sigma'$ connecting the holes coming from $\partial E_0$ and $\partial E$ each.
Similarly to the case of the arc $\alpha$, the arc $C \cap \Sigma'$ is the frontier of a regular neighborhood of the union of $\beta$ and the hole coming from $\partial E_0$ which $\beta$ ends in.

If $\beta$ intersects $\partial E'$, then we repeat the band sum constructions along the sub-arc of $\beta$ connecting $\partial E_0$ and $\partial E'$ to find a dual disk $\widehat{E}'$ of $E$ such that $\widehat{E}'$ is disjoint from $\beta$ and from $\partial E_0$.
Then the arc $C \cap \Sigma'$ is also disjoint from $\widehat{E}'$.
One of the disks from surgery on $E$ along $C$ is $E_0$, and the other one, say $E_1$, intersects $\widehat{E}'$ in a single point.
See Figure \ref{fig_4_holed_sphere_2} (a).
That is, $E_1$ is a primitive disk in $V$ with the dual disk $\widehat{E}'$.
The disk $E_1$ is disjoint from $E_0$, and further we have $|D \cap E_1|<|D \cap E|$.
We repeat the process to find a new primitive disk disjoint from $D$, and also from $E_0$, which has a dual disk disjoint from $E'_0$.
Then we go back to the first case.

\begin{figure}
\centering
\labellist
\pinlabel {\footnotesize $\partial E$} [B] at 113 207
\pinlabel {\footnotesize $\partial E$} [B] at 112 154
\pinlabel {\footnotesize $\partial \widehat{E}$} [B] at 368 207
\pinlabel {\footnotesize $\partial \widehat{E}$} [B] at 368 153
\pinlabel {\small $\partial E_0$} [B] at 113 98
\pinlabel {\small $\partial E_0$} [B] at 368 98
\pinlabel {\small $\partial E_0$} [B] at 193 58
\pinlabel {\small $\partial E_0$} [B] at 450 58
\pinlabel {\scriptsize $\partial \widehat{E}'$} [B] at 123 173
\pinlabel {\scriptsize $\partial \widehat{E}'$} [B] at 379 173
\pinlabel {\footnotesize $C \cap \Sigma'$} [B] at 125 128
\pinlabel {\footnotesize $C_0 \cap \Sigma'$} [B] at 432 134
\pinlabel {\small $\partial E_1$} [B] at 182 119
\pinlabel {\tiny $x$} [B] at 363 189
\pinlabel {\tiny $y$} [B] at 345 95
\pinlabel {\tiny $y^{-1}$} [B] at 395 95
\pinlabel {\large (a)} [B] at 118 5
\pinlabel {\large (b)} [B] at 378 5
\endlabellist
\includegraphics[width=1 \textwidth]{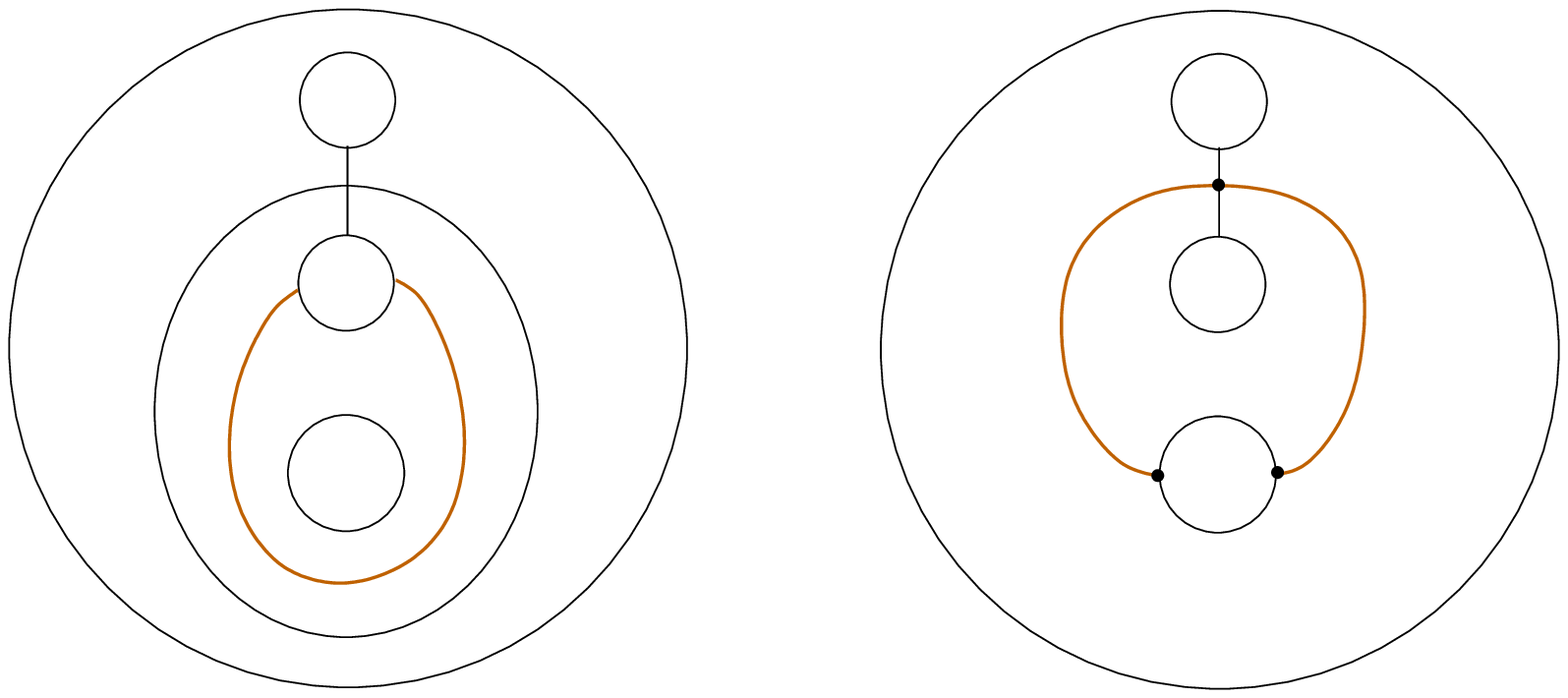}
\caption{
(a) The $4$-holed sphere $\Sigma'$. (b) The $4$-holed sphere obtained by cutting $\Sigma$ along $\partial \widehat{E}$ and $\partial E_0$.}
\label{fig_4_holed_sphere_2}
\end{figure}

\smallskip

\noindent (2) Suppose that $D$ intersects $E_0$.
Let $C_0$ be an outermost sub-disk of $D$ cut off by $D \cap E_0$.
If $C_0$ intersects $E$, the same argument of (1) for the sub-disk $C_0$ (instead of the disk $D$ in (1)) enables us to find a new primitive disk $\widehat{E}$ and its dual disk $\widehat{E}'$ such that $\widehat{E}$ and $\widehat{E}'$ are disjoint from $E_0$ and $E_0'$, and $C_0$ is disjoint from $\widehat{E}$ and intersects $\widehat{E}'$ in a single point.

Giving symbols $x$ and $y$ to the oriented circles $\partial \widehat{E}'$ and $\partial E'_0$,
respectively, the boundary circle $\partial D$ of $D$ determines a word in terms of $x$ and $y$.
In particular, there exists a word determined by $\partial D$ containing $yxy^{-1}$ which is determined by the sub-arc $C_0 \cap \Sigma'$ (after changing orientations if necessary).
See Figure \ref{fig_4_holed_sphere_2} (b).
By Lemma \ref{lem:non_primitive}, $D$ is neither a reducing disk nor a primitive disk.
\end{proof}

In the proof of Lemma \ref{lem:non_primitive_2}, we see that if a non-separating disk $D$ in $V$ is not isotopic to $E_0$, then $\partial D$ represents a non-trivial element of $\pi_1(W)$.
Thus the reducing disk $E_0$ is the unique non-separating disk in $V$ such that $\partial E_0$ represents a trivial element of $\pi_1(W)$ up to isotopy.
The following is also a direct consequence of Lemma \ref{lem:non_primitive_2}.

\begin{lemma}
There exists a unique non-separating reducing disk $E_0$ in $V$.
A non-separating disk $E$ in $V$ is primitive if and only if $E$ is not isotopic to and disjoint from $E_0$.
\label{lem:reducing_disk_and_primitive_disk}
\end{lemma}

Of course, the same result we have for the reducing disk and a primitive disk in $W$.

\begin{lemma}
Any primitive pair has a common dual disk.
\label{lem:common_dual}
\end{lemma}

\begin{proof}
Let $\{D, E\}$ be a primitive pair of $V$, and let $E_0$ and $E'_0$ be the unique reducing disks in $V$ and $W$ respectively such that $\partial E_0 = \partial E'_0$.
Any dual disk $E'$ of $E$ is primitive in $W$, and hence is disjoint from $E'_0$ by Lemma \ref{lem:non_primitive_2}.
The primitive disk $D$ is disjoint from $E_0$ and $E$, thus as in the proof of Lemma \ref{lem:non_primitive_2} (a), we can find a common dual disk $\widehat{E}'$ of $D$ and $E$ by repeating the band sum constructions.
\end{proof}

\begin{lemma}
Let $E'$ be a common dual disk of a primitive pair $\{D, E\}$ of $V$.
Then there exist exactly two common dual disks, say $D'$ and $D''$, of $\{D, E\}$ disjoint from $E'$.
Further, $D'$ intersects $D''$ in a single arc.
\label{lem:two_common_duals}
\end{lemma}

\begin{proof}
Let $E_0$ and $E'_0$ be the unique reducing disks in $V$ and $W$, respectively,
such that $\partial E_0 = \partial E'_0$,
and let $\Sigma''$ be the $4$-holed sphere $\Sigma$ cut off by $\partial E' \cup \partial E'_0$.
Then $\partial E$ and $\partial D$ are disjoint arcs properly
embedded in $\Sigma'$ connecting the two holes coming from $\partial E'$.
The boundary circle of any common dual disk disjoint from $E'$ also lies in $\Sigma''$ and intersects each of $\partial E$ and $\partial D$ in a single point.
Thus there exist exactly two such circles $\partial D'$ and $\partial D''$ illustrated in Figure \ref{fig_common_dual}, which bound two disks intersecting each other in a single arc.
\end{proof}

\begin{figure}
\centering
\labellist
\pinlabel {\scriptsize $\partial E'$} [B] at 42 95
\pinlabel {\scriptsize $\partial E'$} [B] at 157 96
\pinlabel {\small $\partial E'_0$} [B] at 98 95
\pinlabel {\small $\partial E'_0$} [B] at 209 95
\pinlabel {\small $\partial E$} [B] at 52 131
\pinlabel {\small $\partial D$} [B] at 151 60
\pinlabel {\small $\partial D'$} [B] at 70 30
\pinlabel {\small $\partial D''$} [B] at 138 155
\endlabellist
\includegraphics[width=0.45\textwidth]{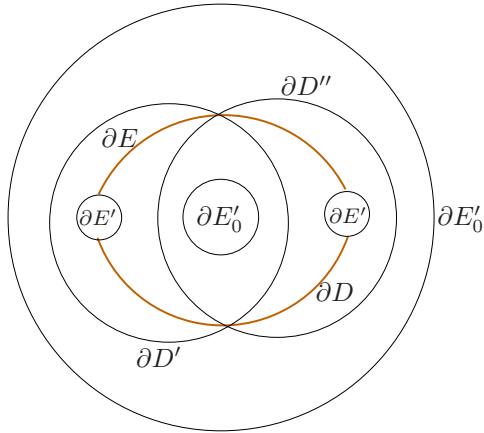}
\caption{
The boundary circles of the two common dual disks $D'$ and $D''$ disjoint from $E'$ in the $4$-holed sphere $\Sigma''$.}
\label{fig_common_dual}
\end{figure}

\section{The complex of primitive disks}
\label{sec:primitive_disk_complex}

Let $M$ be an irreducible $3$-manifold with compressible boundary.
The {\it disk complex} $\mathcal K(M)$ of $M$ is a simplicial complex defined as follows.
The vertices of $\mathcal K(M)$ are isotopy classes of essential disks in $M$, and a collection of $k+1$ vertices spans a $k$-simplex if and only if it admits a collection of representative disks which are pairwise disjoint.
Let $\mathcal D(M)$ be the full subcomplex of $\mathcal K(M)$ spanned by the vertices of non-separating disks, which we will call the {\it non-separating disk complex} of $M$.
It is well known that $\mathcal K(M)$ and $\mathcal D(M)$ are both contractible.
For example, see Theorems 5.3 and 5.4 in \cite{McC}.

In particular, for a genus-$2$ handlebody $V$, the disk complex $\mathcal K(V)$ is a $2$-dimensional simplicial complex which is not locally finite.
Further, we have a precise description of the non-separating disk complex $\mathcal D(V)$ of $V$ as follows.
First, it is easy to verify that $\mathcal D(V)$ is $2$-dimensional, and every edge of $\mathcal D(V)$ is contained in infinitely but countably many $2$-simplices.
Next, $\mathcal D(V)$ itself is contractible, and also the link of any vertex of $\mathcal D(V)$ is a tree each of whose vertices has infinite valency.
This is a direct consequence of Theorem 4.2 in \cite{C}.
Figure \ref{fig_disk_complex} illustrates a portion of $\mathcal D(V)$.
We note that the entire disk complex $\mathcal K(V)$ of $V$ is constructed by attaching infinitely (but countably) many $2$-simplices to each edge of $\mathcal D(V)$, where each of the new vertices is represented by an essential separating disk in $V$.

\begin{figure}
\centering
\labellist
\pinlabel {\small $E_0$} [B] at 108 28
\endlabellist
\includegraphics[width=0.5\textwidth]{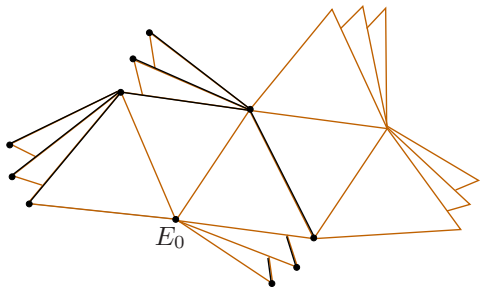}
\caption{A portion of $\mathcal D(V)$. The primitive disk complex $\mathcal P(V)$ is the link of the vertex of the unique reducing disk $E_0$.}
\label{fig_disk_complex}
\end{figure}

Given a genus-$2$ Heegaard splitting $(V, W; \Sigma)$ of $\mathbb S^2 \times \mathbb S^1$, we define the {\it primitive disk complex} $\mathcal P(V)$ to be the full sub-complex of $\mathcal D(V)$ spanned by the vertices of primitive disks in $V$.
By Lemma \ref{lem:reducing_disk_and_primitive_disk}, $\mathcal P(V)$ is the link in $\mathcal D(V)$ of the vertex represented by the unique reducing disk $E_0$ in $V$ (see Figure \ref{fig_disk_complex}), which implies our key result.

\begin{theorem}
The primitive disk complex $\mathcal P(V)$ is an infinite tree each of whose vertices has infinite valency.
\label{thm:primitive_disk_complex}
\end{theorem}

We note here that in \cite{C}, \cite{C2}, \cite{C-K}, \cite{C-K2} and \cite{Kod},
primitive disk complexes are defined in the same way under various settings, and
they are used to obtain presentations of Goeritz groups or their generalizations of given manifolds.

Obviously, the primitive disk complex $\mathcal P(W)$ of $W$ is isomorphic to $\mathcal P(V)$.
Given a primitive disk $E$ of $V$, define the sub-complex $\mathcal P_{\{E\}}(W)$ of $\mathcal P(W)$ to be the full sub-complex spanned by the vertices of dual disks of $E$.
Similarly, $\mathcal P_{\{D, E\}}(W)$ is the full sub-complex of $\mathcal P(W)$ spanned by the vertices of common dual disks of a primitive pair $\{D, E\}$ of $V$.

\begin{theorem}
The sub-complexes $\mathcal P_{\{E\}}(W)$ and $\mathcal P_{\{D, E\}}(W)$ are both infinite trees.
Each vertex of $\mathcal P_{\{E\}}(W)$ has infinite valency, and each vertex of $\mathcal P_{\{D, E\}}(W)$ has valency exactly two.
\label{thm:common_dual_disk_complex}
\end{theorem}

\begin{proof}
It is clear that $\mathcal P_{\{E\}}(W)$ is an infinite sub-complex of the tree $\mathcal P(W)$ whose vertices have infinite valency.
By Lemmas \ref{lem:common_dual} and \ref{lem:two_common_duals}, the sub-complex $\mathcal P_{\{D, E\}}(W)$ is also an infinite sub-complex of the tree $\mathcal P(W)$, and each vertex of $\mathcal P_{\{D, E\}}(W)$ has valency two.
Thus it remains to show that both of $\mathcal P_{\{E\}}(W)$ and $\mathcal P_{\{D, E\}}(W)$ are connected.

Let $E'$ and $D'$ be any two non-isotopic dual disks of $E$.
If $E'$ is disjoint from $D'$, then the two vertices represented by them are joined by a single edge.
If $E'$ intersects $D'$, then it is easy to see that one of the disks from surgery on $E'$ along an outermost sub-disk of $D'$ cut off by $E' \cap D'$ is the unique reducing disk $E'_0$ in $W$, and the other one, say $E''$ is still a dual disk of $E$.
Notice that $E''$ is disjoint from $E'$ and $|D' \cap E''| < |D' \cap E'|$.
By repeating surgery construction, we find a finite sequence of dual disks of $E$ from $E'$ to $D'$, which realizes a path in $\mathcal P_{\{E\}}(W)$ joining the two vertices of $D'$ and $E'$.
Thus $\mathcal P_{\{E\}}(W)$ is connected.
The connectivity of $\mathcal P_{\{D, E\}}(W)$ also can be shown in a similar way, by considering surgery on a common dual disk.
\end{proof}

\begin{remark}
It is interesting to compare the sub-complexes $\mathcal P_{\{D, E\}}(W)$ of common dual disks for genus-$2$ Heegaard splittings of the $3$-sphere $\mathbb S^3$, the lens spaces $L(p, q)$, $p \geq 2$, and $\mathbb S^2 \times \mathbb S^1$.
The following results are known from \cite{C} and \cite{C-K}.
In the case of $\mathbb S^3$, there exist infinitely many common dual disks, and each two of them intersect each other.
Thus we see that $\mathcal P_{\{D, E\}}(W)$ is a collection of infinitely many vertices.
For the lens space $L(p, 1)$, if $p = 2$, then there exist exactly two common dual disks disjoint from each other.
Thus $\mathcal P_{\{D, E\}}(W)$ is a single edge.
If $p \geq 3$, then there exists a unique common dual disk, and so $\mathcal P_{\{D, E\}}(W)$ is a single vertex.
For the other lens spaces, $\mathcal P_{\{D, E\}}(W)$ is either a vertex or the empty set depending on the choice of $\{D, E\}$, and both exist infinitely many.
\end{remark}

\section{The genus two Goeritz group of $\mathbb S^2 \times \mathbb S^1$}
\label{sec:presentations}

In this section, we prove the main theorem, Theorem \ref{thm:presentation}.
We know that the primitive disk complex $\mathcal P(V)$ is a tree from Theorem \ref{thm:primitive_disk_complex}, which is the link in $\mathcal D(V)$ of the vertex of the unique reducing disk $E_0$.
Let $\mathcal T$ be the barycentric subdivision of $\mathcal P(V)$, which is a bipartite tree.
In Figure \ref{fig_tree}, the black vertices of $\mathcal T$ are the vertices of $\mathcal P(V)$, while the white ones are the barycenters of the edges of $\mathcal P(V)$.
The black vertices correspond to the primitive disks, and the white ones to the primitive pairs in $V$ consisting of the primitive disks representing the two adjoining black vertices.
For convenience, we will not distinguish disks (or pairs of disks) and homeomorphisms from their isotopy classes.
\begin{figure}
\centering
\includegraphics[width=0.45\textwidth]{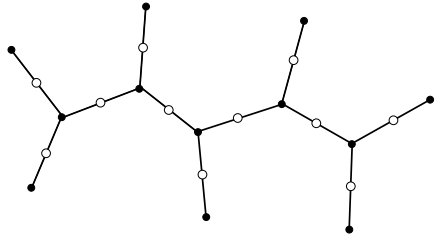}
\caption{
A portion of the tree $\mathcal T$, which is the barycentric subdivision of $\mathcal P(V)$.
The black vertices have infinite valency while white ones have valency two.}
\label{fig_tree}
\end{figure}
The group $\mathcal G$ acts on the tree $\mathcal T$ as a simplicial automorphism, and further we have the following.

\begin{lemma}
The group $\mathcal G$ acts transitively on each of the collections of black vertices and of white vertices of $\mathcal T$.
\label{lem:transitive}
\end{lemma}

\begin{proof}
The baricentric subdivision $\mathcal T$ of $\mathcal P(V)$ is a tree and so connected.
Thus, given any two primitive disks in $V$, there exists a sequence of primitive disks from one to another in which any two consecutive disks form a primitive pair in $V$.
Therefore, to see the transitivity of the action of $\mathcal G$ on the black vertices, it suffices to find an element of $\mathcal G$ sending a disk $D$ to a disk $E$ for any primitive pair $\{D, E\}$ in $V$.
Such an element does exist, since we have the element $\sigma$ in Section \ref{sec:primitive_disks}. (The element $\sigma$ sends $D$ to $E$ for the primitive pair $\{D, E\}$ in $V$.
In fact, $\sigma$ exchanges $D$ and $E$.)

Next, since $\mathcal T$ is connected again, given any two primitive pairs in $V$, there exists a sequence of primitive pairs from one to another in which any two consecutive pairs share a common primitive disk.
Thus, to see the transitivity of the action of $\mathcal G$ on the white vertices, it suffices to find an element of $\mathcal G$ sending a disk $D$ to a disk $F$ and fixing a disk $E$ for any two primitive pairs $\{D, E\}$ and $\{E, F\}$ in $V$.

Choose any common dual disks $D'$ and $F'$ of the pairs $\{D, E\}$ and $\{E, F\}$ respectively.
Since $\mathcal P_{\{E\}}(W)$ is also a tree by Theorem \ref{thm:common_dual_disk_complex}, there exists a sequence $D' = D'_0, D'_1,$ $\cdots,$ $D'_n = F'$ of dual disks of $E$ in which any two consecutive disks form a primitive pair in $W$.
Let $h_i$ be an element of $\mathcal G$ sending $D'_{i-1}$ to $D'_i$ and fixing $E$.
Such an element does exist, since we have the element $\gamma$ in Section \ref{sec:primitive_disks}. 
(In the description, the element  $\gamma$ fixes $E$ and exchanges two disjoint dual disks $D'$ and $E'$ of $E$.)
Then the composition $h = h_k h_{k-1} \cdots h_1$ sends $D'$ to $F'$ and fixes $E$.
We observe that $\{h(D), E\}$ and $\{E, F\}$ are primitive pairs and $F'$ is a common dual disk of them.
Then there exists an element $g$ of $\mathcal G$ sending $h(D)$ to $F$ and fixing $E$.
(The element $g$ will be a power of the element $\beta$ in Section \ref{sec:primitive_disks}, when we take $h(D)$, $E$ in $V$ and $F'$ in $W$ instead of $D'$, $E'$ in $W$  and $E$ in $V$ respectively.)
Then the composition $gh$ sends $D$ to $F$ and fixes $E$ as desired.
\end{proof}

From the lemma, we see that the quotient of $\mathcal T$ by the action of $\mathcal G$ is a single edge whose one vertex is black and another one white.
By the theory of groups acting on trees due to Bass and Serre \cite{S}, the group $\mathcal G$ can be expressed as the free product of the stabilizer subgroups of two end vertices with amalgamated stabilizer subgroup of the edge.

Throughout the section, $\mathcal G_{\{A_1, A_2, \cdots, A_k\}}$ will denote the subgroup of $\mathcal G$ of elements preserving each of $A_1, A_2, \cdots, A_k$ setwise, where $A_i$ will be (the isotopy classes of) disks or union of disks in $V$ or in $W$.
Then, given a primitive pair $\{D, E\}$ of $V$, the subgroups $\mathcal G_{\{E\}}$, $\mathcal G_{\{E \cup D\}}$ and $\mathcal G_{\{E, D\}}$ are the stabilizer subgroups of a black vertex, of a white vertex and of the edge of $\mathcal T$ joining them respectively.
Thus the Goeritz group $\mathcal G$ is the free product of $\mathcal G_{\{E\}}$ and $\mathcal G_{\{E \cup D\}}$ amalgamated by $\mathcal G_{\{E, D\}}$.

We first find presentations of the three stabilizer subgroups $\mathcal G_{\{E\}}$, $\mathcal G_{\{E \cup D\}}$ and $\mathcal G_{\{E, D\}}$.
As mentioned above, the basic idea is to describe the generators of the subgroups as homeomorphisms of the surface $\Sigma$ preserving the boundary circles of some primitive disks and reducing disks in $V$ and $W$.

\begin{lemma}
The stabilizer subgroup $\mathcal G_{\{E\}}$ has the presentation
$$\langle ~\epsilon~ \rangle \oplus \langle ~\alpha ~|~ \alpha^2 =1~\rangle \oplus \langle ~\beta, \gamma ~|~ \gamma^2 = 1 ~ \rangle.$$
\label{lem:black_stabilizer}
\par
\end{lemma}

\begin{proof}
By Theorem \ref{thm:common_dual_disk_complex}, $\mathcal P_{\{E\}}(W)$ is the sub-tree of the tree $\mathcal P(W)$ spanned by the vertices of the dual disks of $E$.
The barycentric subdivision of $\mathcal P_{\{E\}}(W)$, which we denote by $\mathcal T_{\{E\}}$, is a bipartite tree, each of whose vertices corresponds to either a dual disk of $E$ or a pair of disjoint dual disks of $E$.
The subgroup $\mathcal G_{\{E\}}$ acts on $\mathcal T_{\{E\}}$, and its quotient is again a single edge.
Thus, fixing a primitive pair $\{D', E'\}$ of dual disks of $E$ as in Figure \ref{fig_generators}, the subgroup $\mathcal G_{\{E\}}$ can be expressed as the free product of stabilizer subgroups $\mathcal G_{\{E, E'\}}$ and $\mathcal G_{\{E, D' \cup E'\}}$ amalgamated by $\mathcal G_{\{E, D', E'\}}$.

First, consider the subgroup $\mathcal G_{\{E, E'\}}$.
This group also preserves each of $E_0$ and $E'_0$, and so $\mathcal G_{\{E, E'\}} = \mathcal G_{\{E, E', E_0, E_0'\}}$.
This group is generated by the elements $\beta$, $\beta'$ and $\epsilon$.
Thus we have a presentation of $\mathcal G_{\{E, E'\}}$ as
$\langle ~\epsilon ~ \rangle \oplus \langle ~\beta, \beta' ~|~ (\beta \beta')^2 = 1, ~\beta \beta' = \beta' \beta ~ \rangle$.

Next, the subgroup $\mathcal G_{\{E, D' \cup E'\}} = \mathcal G_{\{E, D' \cup E', E_0, E'_0\}}$ is generated by $\alpha$, $\gamma$ and $\epsilon$, and so it has the presentation
$\langle ~\epsilon ~ \rangle \oplus \langle ~\alpha ~|~ \alpha^2 =1~\rangle \oplus \langle ~\gamma ~|~ \gamma^2 = 1 ~ \rangle$.
In a similar way, the index-$2$ subgroup $\mathcal G_{\{E, D', E'\}}$ of $\mathcal G_{\{E, D' \cup E'\}}$ has the presentation
$\langle ~\epsilon ~ \rangle \oplus \langle ~\alpha ~|~ \alpha^2 =1~\rangle$.
Observing $\beta \beta' = \alpha$, we have the desired presentation of $\mathcal G_{\{E\}}$.
\end{proof}

\begin{lemma}
The stabilizer subgroup $\mathcal G_{\{D \cup E\}}$ has the presentation
$$\langle ~\epsilon~ \rangle \oplus \langle ~\alpha ~|~ \alpha^2 =1~\rangle \oplus \langle ~\sigma, \tau ~|~ \sigma^2 = 1, ~(\tau \sigma)^2 = 1 ~ \rangle.$$
\label{lem:white_stabilizer}
\par
\end{lemma}

\begin{proof}
By Theorem \ref{thm:common_dual_disk_complex} again, the full subcomplex $\mathcal P_{\{D\cup E\}}(W)$ of $\mathcal P(W)$ spanned by the vertices of common dual disks of the pair $\{D, E\}$ is a tree.
Denote by $\mathcal T_{\{D \cup E\}}$ the barycentric subdivision of $\mathcal P_{\{D\cup E\}}(W)$.
Then the quotient of $\mathcal T_{\{D \cup E\}}$ by the action of $\mathcal G_{\{D \cup E\}}$ is a single edge.
Thus, fixing a primitive pair $\{D', E'\}$ of common dual disks of $D$ and $E$, the subgroup $\mathcal G_{\{D \cup E\}}$ can be expressed as the free product of $\mathcal G_{\{D \cup E, E'\}}$ and $\mathcal G_{\{D \cup E, D' \cup E'\}}$ amalgamated by $\mathcal G_{\{D \cup E, D', E'\}}$.

Similarly to the case of $\mathcal G_{\{E, D' \cup E'\}}$ in the proof of Lemma \ref{lem:black_stabilizer}, $\mathcal G_{\{D \cup E, E'\}}$ is generated by $\alpha$, $\sigma$ and $\epsilon$, and hence has the presentation
$\langle ~\epsilon ~ \rangle  \oplus  \langle ~\alpha ~|~ \alpha^2 =1~\rangle \oplus \langle ~\sigma ~|~ \sigma^2 = 1 ~ \rangle$.
Next, $\mathcal G_{\{D \cup E, D', E'\}}$ is a subgroup of $\mathcal G_{\{D \cup E, E'\}}$, which does not contain $\sigma$ since $\sigma$ does not preserve $D'$.
Thus $\mathcal G_{\{D \cup E, D', E'\}}$ has the presentation $\langle ~\epsilon ~ \rangle  \oplus \langle ~\alpha ~|~ \alpha^2 =1~\rangle$.
Finally, recall the order two element $\tau' = \tau \sigma$ of $\mathcal G_{\{D \cup E, D' \cup E'\}}$ exchanging $D$ and $E$, and $D'$ and $E'$ respectively.
The subgroup $\mathcal G_{\{D \cup E, D' \cup E'\}}$ is the extension of $\mathcal G_{\{D \cup E, D', E'\}}$ by $\tau'$, so it has the presentation
$\langle ~\epsilon ~ \rangle  \oplus   \langle ~\alpha ~|~ \alpha^2 =1~\rangle \oplus \langle ~\tau' ~|~ \tau'^2 = 1 ~ \rangle$, and so the desired presentation of $\mathcal G_{\{D \cup E\}}$ is obtained.
\end{proof}

\begin{lemma}
The stabilizer subgroup $\mathcal G_{\{D,  E\}}$ has the presentation
$$\langle ~\epsilon~ \rangle \oplus \langle ~\tau ~ \rangle \oplus \langle ~\alpha ~|~ \alpha^2 =1 ~ \rangle.$$
\label{lem:edge_stabilizer}
\end{lemma}
\par

\begin{proof}
The presentation of $\mathcal G_{\{D,  E\}}$ is obtained directly from the fact that $\mathcal G_{\{D \cup E\}}$ is the extension of $\mathcal G_{\{D,  E\}}$ by the order-2 element $\sigma$ exchanging $D$ and $E$.
\end{proof}

\begin{remark}
The subgroup $\mathcal G_{\{D,  E\}}$ acts on the tree $\mathcal T_{\{D \cup E\}}$.
Although this action is transitive on the collections of vertices of common dual disks and of pairs of common dual disks of $D$ and $E$ respectively, the quotient is not a single edge.
In fact, one can verify that the quotient is a single loop and can obtain the same presentation of $\mathcal G_{\{D,  E\}}$ to the above using this loop.
\end{remark}

Combining Lemmas \ref{lem:black_stabilizer}, \ref{lem:white_stabilizer} and \ref{lem:edge_stabilizer} and using $\tau = \gamma \beta$,
we obtain Theorem \ref{thm:presentation}.

\medskip

\noindent {\bf Acknowledgments.}
The second-named author wishes to thank Universit\`a di Pisa
for its kind hospitality.
The authors are grateful to the referee for helping them to improve the presentation.

\bibliographystyle{amsplain}

\end{document}